\documentclass{article}

\usepackage{amssymb,latexsym}
\usepackage{amsmath}
\usepackage{amsthm}
\usepackage{graphics}
\usepackage{dsfont}
\usepackage{multicol}
\usepackage{coordsys}
\usepackage{logsys}

\newtheorem{theorem}{Theorem}
\newtheorem{corollary}{Corollary}
\newtheorem{lemma}{Lemma}
\newtheorem{observation}{Observation}

\newtheorem{definition}{Definition}
\newtheorem{example}{Example}
\newtheorem{notation}{Notation}

\numberwithin{equation}{section}

\mathchardef\mhyphen="2D
\newcommand{\hy}{\mhyphen}

\newcommand{\cou}{C_{1}\left(\mathcal{U}\right)}
\newcommand{\coso}[1]{C_{1}\left(#1\right)}
\newcommand{\cnu}{C_{n}\left(\mathcal{U}\right)}
\newcommand{\csu}[1]{C_{{#1}}\left(\mathcal{U}\right)}

\renewcommand{\a}{\mathcal{A}}
\renewcommand{\b}{\mathcal{B}} 
\renewcommand{\k}{\mathcal{K}}
\newcommand{\p}{\mathcal{P}}
\newcommand{\R}{\mathbb{R}}
\newcommand{\Q}{\mathbb{Q}}

\renewcommand{\u}{\mathcal{U}}
\newcommand{\su}{\,$\mathcal{U}$\,\,}
\newcommand{\sa}{\,$\mathcal{A}$\,\,}
\newcommand{\saa}{\,$\mathcal{A}$}
\newcommand{\suu}{\,$\mathcal{U}$}
\newcommand{\uo}{\mathcal{U}_1}
\newcommand{\uoo}{\mathcal{U}_2}
\newcommand{\uooo}{\mathcal{U}_3}

\newcommand{\un}{\mathcal{U}_n}
\newcommand{\uno}{\mathcal{U}_{n+1}}
\newcommand{\ui}{\mathcal{U}_{\infty}}
\newcommand{\us}[1]{\mathcal{U}_{{#1}}}

\newcommand{\vv}{\mathcal{V}}
\newcommand{\sv}{\,$\mathcal{V}$\,}

\newcommand{\h}{\mathcal{H}}
\newcommand{\hn}{\mathcal{H}_n}

\newcommand{\ho}{\mathcal{H}_1}
\newcommand{\hi}{\mathcal{H}_{\infty}}
\newcommand{\hs}[1]{\mathcal{H}_{{#1}}}
\newcommand{\sh}{\,$\mathcal{H}$\,\,}
\newcommand{\shh}{\,$\mathcal{H}$}

\newcommand{\ec}{{E}^{\mathsf{c}}}
\renewcommand{\sc}[1]{{#1}^{\mathsf{c}}}

\newcommand{\ds}[1]{{\displaystyle #1}}
\newcommand{\at}[2][rrrrrrrrrrrrrr]{
                     \begin{array}{#1}
                     #2\\
                     \end{array}}
\newcommand{\defi}[2]{#1= \left\{\at[lllllllllllll]{#2}\right.} %\defi{f(x)}{3 & x\le 4\\ 2 & x>4}
\newcommand{\define}[5]{\defi{\ds{#1}}{\ds{#2} & : \, \mathrm{\ if\ \,} \ds{#3},  \\[3mm]
                         \ds{#4} & : \, \mathrm{\ if\ \, } \ds{#5}.}}
\newcommand{\mini}[2]{\begin{minipage}{#1\textwidth}
            \begin{flushleft}#2\end{flushleft}\end{minipage}}

\newcommand{\sub}{\subseteq} 

\begin{document}
\title{Minimal Constructible Sets}
\author{Jorge Garcia, Tyler Bongers, Jonathan Detgen, Walter Morales}
\maketitle

%%%%%%%%%%%%%%%%%%%%%%%%%%%%%%%%%%%%%%%%%%%%%%%%%%%%%%%%%%%%%%%%%%%%%%%%%%%%%%%%%%%%%%%%%%%%%%%%%%%%%%%%%%%%%%%%%%%%%%%%%%%%%%%%%%%%%%%%%%%%%%%%%%%%%%%%%%%%%%%%%%%%%%%%%%%%%%%%%%%%%%%%%%%%%%%%%%%%%%%%%%%%%%%%%%%%%%%%%%%%%%%%%%%%%%%%%%%%%%%%%%%%%%%%%%%%%%%%%%%%%%%%%%%%%%%%%%%%%%%%%%%%%%%%%%%%%%%%%%%%%%%%%%%%%%%%%%%%%%%%%%%%%%

\begin{abstract}
   Given an initial family of sets, we may take unions, intersections and complements of the sets contained in this family in order to form a new collection of sets; our construction process is done recursively until we obtain the last family. Problems encountered in this research include the minimum number of steps required to arrive to the last family as well as a characterization of that last family; we solve all those problems. We also define a class of simple families ($n$-minimal constructible) and we analyze the relationships between partitions and separability (our new concept) that lead to interesting results such as finding families based on partitions that generate finite algebras. We prove a number of new results about $n$-minimal constructible families; one major result is that every finite algebra of sets has a generating family which is $n$-minimal constructible for all $n \in \mathds{N}$ and we compute the minimum number of steps required to generate an algebra. Another interesting result is a connection between this construction process and Baire's Theorem. This work has a number of possible applications, particularly in the fields of economics and computer science.
\end{abstract}

%%%%%%%%%%%%%%%%%%%%%%%%%%%%%%%%%%%%%%%%%%%%%%%%%%%%%%%%%%%%%%%%%%%%%%%%%%%%%%%%%%%%%%%%%%%%%%%%%%%%%%%%%%%%%%%%%%%%%%%%%%%%%%%%%%%%%%%%%%%%%%%%%%%%%%%%%%%%%%%%%%%%%%%%%%%%%%%%%%%%%%%%%%%%%%%%%%%%%%%%%%%%%%%%%%%%%%%%%%%%%%%%%%%%%%%%%%%%%%%%%%%%%%%%%%%%%%%%%%%%%%%%%%%%%%%%%%%%%%%%%%%%%%%%%%%%%%%%%%%%%%%%%%%%%%%%%%%%%%%%%%%%%%

\section{Introduction} %%%%%%%%%%%  111111111111111111
   This paper examines the results of an iterative process of set construction using elementary set operations. Given a family \su of subsets of some universe $X$, we wish to characterize what can be built from \su by successive operations of union, intersection and complement. Interesting problems encountered in this work include finding methods of characterizing constructible sets and simple generating families for a given complicated family. \\
   \indent By examining the construction process and the forms of the families constructed, we have been able to prove several general theorems giving a class of simple generating families for any finite algebra. In particular, we have proven in this paper that every finite algebra has a generating family which is 1-minimal constructible; we have examined other properties of these families, and have determined how many steps are necessary to reach a last family (in which no new sets can be constructed) from our 1-minimal constructible generating families. Connections between this construction process and topology (Baire's Theorem) and measurable sets are presented through some theorems regarding non-constructible sets, specifically, we prove that the rationals can not be constructed from the usual topology for the real numbers.

  The concept of \textit{constructibility} has been studied in a topological sense: a set is constructible if it belongs to the smallest algebra that contains the open intervals, see for example Chevalley's Theorem: under certain conditions the image under a polynomial map of a constructible set is constructible ~\cite{Marker2002}. Equivalently, a set is constructible if and only if it is a finite union of locally closed sets, that is, the intersection of a closed set and an open set (see Lemma~\ref{lmmuinftopology}). Particularly when we start with a topology, in~\cite{Allouche1996}, Allouche defines an operator on a set and concludes that a set is constructible if and only iff there is $n$ such that the defined operator iterated $n$ times on this set produces an empty set, in fact, he provides a canonical decomposition of the set into $n$ locally closed sets, in other words, Allouche tells you how construct this set when the set is in indeed constructible. In a forthcoming paper~\cite{GaBaMoBoPoCe}, we also analyze the decomposition of constructible sets.   When our initial family \su consists of the open sets, this topological definition matches ours, hence our definition is a generalization of the existing one. In fact, we characterize the constructible sets when \su consists of the open sets, when \su consists of the open intervals and when \su consists of a finite partition. Therefore many of our results apply in the topological case.

The organization of the paper is as follows. In Section~\ref{SectionDefinitionsExamples}, we provide the basic definitions and some examples to illustrate future tools and in Section~\ref{SectionElementaryResults} we study elementary results about algebras, we define our construction and ``the last family". These results will be needed through the whole paper. It is in Section~\ref{SectionNonConstructibleAndCharacterization} where we provide examples of sets that are not constructible and we characterize which sets are constructible in different cases, these are when the initial family is the open intervals and when the initial family is the usual topology. In addition, we provide a theorem that contrasts our finding with the characterization of a sigma-algebra. In Section~\ref{SectionSeparabilityAndPartitons}, we state the main concepts that give structure to our main results; particularly, we develop the concept of separability and analyze in depth the relationship between this concept and partitions in our construction process. Many of these results hold even if our universe set is infinite. Here we also define the minimum number of steps required to generate an algebra and we prove an important result of this paper, namely, we compute this number given any finite algebra. Finally in Section~\ref{SectionMinimalConstructibility}, we analyze the concept of $n$-minimal constructibility and we prove two of the main results of this paper pertaining to the obtention of an $n$-minimal constructible family that generates a finite algebra in a minimum number of steps, we also compute this number. In Section~\ref{SectionFutureWork}, we propose future work that extends some results that have been done here including counting certain number of families.  

We would like to remark that often we provide some trivial examples because these examples inspire the definitions that follow them.

%%%%%%%%%%%%%%%%%%%%%%%%%%%%%%%%%%%%%%%%%%%%%%%%%%%%%%%%%%%%%%%%%%%%%%%%%%%%%%%%%%%%%%%%%%%%%%%%%%%%%%%%%%%%%%%%%%%%%%%%%%%%%%%%%%%%%%%%%%%%%%%%%%%%%%%%%%%%%%%%%%%%%%%%%%%%%%%%%%%%%%%%%%%%%%%%%%%%%%%%%%%%%%%%%%%%%%%%%%%%%%%%%%%%%%%%%%%%%%%%%%%%%%%%%%%%%%%%%%%%%%%%%%%%%%%%%%%%%%%%%%%%%%%%%%%%%%%%%%%%%%%%%%%%%%%%%%%%%%%%%%%%%%

\section{Definitions and Examples} %%%%%%%%%%%  2222222222222222222222
\label{SectionDefinitionsExamples}

Consider $X$ any nonempty set (our universe) and \,$\u$ a non-empty family of subsets of $X$. Throughout this paper, Script letters such as \suu,\,\sa or \sh will represent collections, or families, of subsets of the universe $X$. The number $|\u|$ will denote the cardinality of \su and the number $n$ will always denote a non-negative integer.

\begin{definition}\label{defnconstruct}
   A subset $A$ of $X$ is 1-constructible or 1-\,$\u$ constructible if and only if there are two elements $E_1, E_2$ in \,$\u$ such that either $A=E_1\cap E_2,\ A=E_1\cup E_2$ or $A=\ec_1.$
   We denote the family of 1-constructible sets by
   \begin{align*}
     \cou=\left\{A\subseteq X\ : \ A \mathrm{\ is \ 1\hy}\,\u \mathrm{\ constructible} \right\}.
   \end{align*}
\end{definition}

Intuitively, $\cou$ is the family of subsets of $X$ that can be constructed from \,$\u$ in one step using either the complement of an element in \,$\u$ or the union or intersection of two elements in \,$\u$.
We denote by $\csu{2}$ the family $\coso{\cou}$ and in general for $n\ge 1,$
    \begin{align*}
     \cnu=\coso{\csu{n-1}}.
   \end{align*}

To shorten notation, we will use indistinctively $\un$ instead of $\cnu$ with the clear understanding that $\us{0}=\csu{0}=\u.$

\begin{definition}\label{defnlastfamily}
  We define the last family $C_\infty(\u)$ or $\us{\infty}$ by
    \begin{align*}
     \us{\infty}=\bigcup_{n=1}^\infty \un.
   \end{align*}
\end{definition}

\begin{definition}\label{defnnsteps}
   We say that a subset $A$ of $X$ is $n$-constructible or $n\hy\u$ constructible if and only if $A \in  \un$.   We also say $A$ is constructible or \,$\u$ constructible if and only if $A \in  \us{\infty}$.
\end{definition}

\begin{example}
  Consider $X = \{1,2,3\}$ and $\mathcal{U} = \{\{1,2\},\{2,3\}\}$. We have that
  \begin{align*}
    \uo &= \{\{1,2\},\{2,3\},X,\{2\},\{1\},\{3\}\} = \u \cup \{\{1\},\{2\},\{3\},X\}, \\
    \uoo &= \{\{1,2\},\{2,3\},X,\{2\},\{1\},\{3\},\{1,3\},\emptyset\} = \uo \cup \{\{1,3\},\emptyset\}, \\
    \uooo &= \{\{1,2\},\{2,3\},X,\{2\},\{1\},\{3\},\{1,3\},\emptyset\} = \uoo, \\
    \ui &= \uooo  = \uoo.
  \end{align*}
\end{example}

Notice that $|\mathcal{U}|=2$ and $|\ui|=2^3.$ It is also important to note that if a set $A$ is in $\ui$, then there is $n\ge 1$ such that $A \in \un$. Thus, $\ui$ is the collection of all sets constructible from $\u $ in a finite sense, i.e., its elements can be constructed  in terms of only finite unions, finite intersections and complements.

\begin{definition}\label{defnAlgebra}
   A family $\mathcal{A}$ of subsets of $X$ is an algebra of sets  if and only if \,$\a$ is closed under complements, finite unions and finite intersections. If the family is also closed under countable unions, it is called a $\sigma-$algebra.
\end{definition}

The following lemma follows immediately from Definitions~\ref{defnlastfamily} and ~\ref{defnAlgebra}.

\begin{lemma}\label{LemmaUiIsAlgebra}
  $\ui$ is an algebra of sets, it contains $\u$ and hence $\ui$ is the algebra generated by $\u$ (the samlles algebra that contains $\u$).
\end{lemma}

  Although we place emphasis on finite construction, we do not require finite families or finite universes. The following two examples will lead to interesting results later on in this paper. In the first example we consider \,$\u$ a partition of $X$ and in the second one we consider an infinite family \,$\u.$

\begin{example}
  Consider $X = \{1,2,3,4,5,6\}$ and $\u  = \{\{1\},\{2,5\},\{3,4\},\{6\}\}$. Then we have
  \begin{align*}
    \u &= \{\{1\},\{2,5\},\{3,4\},\{6\}\},   \\
    \uo  &= \u  \cup \{\{1,2,5\},\{1,3,4\},\{1,6\},\{2,3,4,5\},\{2,5,6\}, \\
    								& \{3,4,6\},\emptyset\,\{2,3,4,5,6\}, \{1,3,4,6\},\{1,2,5,6\}, \{1,2,3,4,5\}\}, \\
    \ui &=\uoo= \uo  \cup \{X\}.
  \end{align*}
  Notice that every set in $\mathcal{U}_{\infty}$ consists of unions of sets in the partition; In this paper, we will show that this holds in general when $\u$ forms a partition. Note also that, since $\mathcal{U}_{\infty}$ consists only of  unions of elements in $\u $, we have $\left|\mathcal{U}_{\infty}\right| = 2^{\left|\u \right|}$.
\end{example}
%
%\begin{example}
%   Consider the initial family $ \u = \left\{ \left[0 , \infty \right) \right\}$ and $X= \R $. It follows that $ \u_{2} = \u_{1} \cup \left\{ \left( \infty, 0 \right) \right\}$. Then, following this process the only new elements generated in $\u_{3}$ will be $ \emptyset$ and $\mathds{R}$; the construction process does not yield any new sets, so $ \u_{3}$ becomes $ \u_{\infty}$.
%\end{example}

%We now consider an example in which both the universe and initial family are uncountably infinite.

\begin{example}
   Consider the family of all closed intervals \ $\u = \left\{
   \left[a_{1}, b_{1} \right]: a_{1}, b_{1} \in  \R  \right\}$ and $X= \R $ as the
   universe. Here
   \begin{eqnarray*}
      \u_{1} & =  & \u \cup \left\{\emptyset, \left( -
        \infty, a_{1} \right) \cup \left( b_{1}, \infty \right), \left[ a_{2},
        b_{2} \right] \cup \left[ a_{3}, b_{3} \right] : a_{i}, b_{i} \in
         \R  \right\}\\
      \u_{2} & = & \u_{1} \cup \{\R, \, \left[a_{1}, b_{1}
        \right), \left( a_{2}, b_{2} \right], \left[ a_{3}, b_{3} \right) \cup
        \left( a_{4}, b_{4} \right], \left(a_{5}, b_{5} \right] \cup
        \left[a_{6}, b_{6} \right),\\
      & {} & \left[a_{7}, b_{7} \right] \cup
        \left[a_{8}, b_{8} \right), \left(a_{9}, b_{9} \right] \cup  \left[a_{10}, b_{10} \right],
        \left(a_{11}, b_{11} \right] \cup \left[ a_{12}, b_{12} \right] \cup
        \left[a_{13}, b_{13} \right), \\
      & {} &  \left( - \infty, a_{14} \right) \cup
        \left( a_{15}, b_{15} \right) \cup \left( b_{16}, \infty \right),
        \left( - \infty, b_{17} \right) \cup \left[ a_{18}, \infty
        \right), \\
      & {} & \left( - \infty, b_{19} \right] \cup \left( a_{20}, \infty
        \right), \left( - \infty, b_{21} \right) \cup \left[a_{22}, b_{22}
        \right] \cup \left( a_{23}, \infty \right) \ \, : a_{i}, b_{i} \in  \R  \}.
   \end{eqnarray*}
   It follows that $\u_{\infty}$ consists of all finite unions of all types of intervals (closed, open, half-open, half-closed even unbounded intervals).
\end{example}

Thus, the construction process can begin with relatively simple families and generate very complicated examples quickly. In the next section, we prove some elementary lemmas necessary to consider the more complicated cases encountered in this field.

%%%%%%%%%%%%%%%%%%%%%%%%%%%%%%%%%%%%%%%%%%%%%%%%%%%%%%%%%%%%%%%%%%%%%%%%%%%%%%%%%%%%%%%%%%%%%%%%%%%%%%%%%%%%%%%%%%%%%%%%%%%%%%%%%%%%%%%%%%%%%%%%%%%%%%%%%%%%%%%%%%%%%%%%%%%%%%%%%%%%%%%%%%%%%%%%%%%%%%%%%%%%%%%%%%%%%%%%%%%%%%%%%%%%%%%%%%%%%%%%%%%%%%%%%%%%%%%%%%%%%%%%%%%%%%%%%%%%%%%%%%%%%%%%%%%%%%%%%%%%%%%%%%%%%%%%%%%%%%%%%%%%%%

\section{Elementary Results} %%%%%%%%%%%  333333333333333333333333333333
\label{SectionElementaryResults}

In this section, we will present some basic results pertaining to set constructions; This will provide the necessary groundwork for the major results of this paper. For definitions of topology, closed sets, open sets, closure and interior refer to~\cite{RoydenFitzpatrick2010}.

\begin{lemma}\label{LemmaBasicTripple}
  If \su and \sh are two families of subsets of $X$ and $n\ge 0$, then
  \begin{itemize}
    \item $\un\subseteq\uno.$
    \item If  \,$\h \subseteq \u$\, then \,$\hn\sub\un$\, and hence \,$\hi\sub\ui$.
    \item $(\u\cup\h)_n\supseteq\un\cup\hn$\, and \,\,$(\u\cap\h)_n\subseteq\un\cap\hn$.
  \end{itemize}
\end{lemma}

\begin{proof} % Proof of Lemma Triple
  Clearly if $A\in \un$ then $A=A\cup A,$ hence $A\in \uno.$ The second part is obvious as anything that can be constructed from \sh in $n$ steps can also be constructed from \su in $n$ steps because \sh is part of \suu, hence \,$\hn\sub\un$. Taking unions in this collection of sets over all $n$ gives \,$\hi\sub\ui.$ Finally, applying the second part to the collections \,$\u\sub \u\cup\h,\ \h\sub \u\cup \h$ \,  \, \,$\h\cap\u\sub\h,\ \,\h\cap\u\sub\u$ and taking the corresponding operation gives the desired result.
\end{proof}

 We recall that every $\sigma$-algebra is an algebra (see \cite{RoydenFitzpatrick2010}, Section 4, Chapter 1).   The following two lemmas are immediate consequences of Lemma~\ref{LemmaBasicTripple} and Definition~\ref{defnAlgebra}.

\begin{lemma}\label{LemmaAlgebraPreserved}
  If $\a$ is an algebra of sets of $X$ then $\coso{\a}=\a.$
\end{lemma}

\begin{lemma}\label{LemmaAlgebraContains}
  If $\a$ is an algebra of sets with $\a\supseteq\u$ then $\a\supseteq\cnu.$
\end{lemma}

We will now prove two lemmas that will be useful.

\begin{lemma}\label{LemmaUnIsFinite}
  Suppose that \,$\u$ is finite. Then \,$\uo$ has at most $|\u|^2+|\u|$ elements, hence for any $n, \ \un$ is finite. If in addition, there is $n_0$ with $|\us{n_0}|=|\us{n_0+1}|$ then \,$\us{n_0}$ is an algebra of sets and \,$\ui=\us{n_0}$.
\end{lemma}

\begin{proof} % Proof of Lemma LemmaUnIsFinite
   If \,$\u$ has $N$ elements, to build $\uo$ we need $\binom{N}{2}+N$ unions (of two elements in \,$\u$), $\binom{N}{2}$ intersections and $N$ complements. This gives a total of at most $2N+2\binom{N}{2}=N^2+N$ elements in $\uo.$ An inductive argument shows that $\un$ is finite. Now, consider $n_0$ with $|\us{n_0}|=|\us{n_0+1}|$. By Lemma \ref{LemmaBasicTripple}, $\us{n_0}\sub\us{n_0+1}.$ Then, \ $\us{n_0}=\us{n_0+1}$ and hence for all $n\ge n_0, \ \un=\us{n_0},$ hence \,$\ui=\us{n_0}$. By Lemma~\ref{LemmaUiIsAlgebra}, \ $\ui$ is an algebra of sets.
\end{proof}

\begin{lemma}\label{LemmaUiPreserved}
  $\coso{\ui}=\ui$.  If \,$\u$ is finite, then $\ui$ is a finite topology, an algebra and also a $\sigma$-algebra.
\end{lemma}

\begin{proof} % Proof of Lemma LemmaUiPreserved
   $\coso{\ui}=\ui$ comes from Lemma~\ref{LemmaAlgebraPreserved} and Lemma~\ref{LemmaUiIsAlgebra}.
   If \,$\u$ is finite, then $\ui$ is also finite with at most $2^{|X|}$ elements. Therefore any countable (or arbitrary) union of elements in $\ui$ really becomes a finite union of elements in $\ui$, therefore $\ui$ is also closed under countable (or arbitrary) unions, thus $\ui$ is also a sigma algebra and a topology for $X.$
\end{proof}

Armed with these results, we will now examine some more involved examples.

%%%%%%%%%%%%%%%%%%%%%%%%%%%%%%%%%%%%%%%%%%%%%%%%%%%%%%%%%%%%%%%%%%%%%%%%%%%%%%%%%%%%%%%%%%%%%%%%%%%%%%%%%%%%%%%%%%%%%%%%%%%%%%%%%%%%%%%%%%%%%%%%%%%%%%%%%%%%%%%%%%%%%%%%%%%%%%%%%%%%%%%%%%%%%%%%%%%%%%%%%%%%%%%%%%%%%%%%%%%%%%%%%%%%%%%%%%%%%%%%%%%%%%%%%%%%%%%%%%%%%%%%%%%%%%%%%%%%%%%%%%%%%%%%%%%%%%%%%%%%%%%%%%%%%%%%%%%%%%%%%%%%%%

\section{Some Non-Constructible Sets and Their Characterizations} %%%%%%%%%%%  444444444444444
\label{SectionNonConstructibleAndCharacterization}

\indent Through the construction process, only certain sets are constructible. In this section, we present a proof that the set $ \Q $ of rational numbers is not finitely constructible from the collection of open intervals. This example does provide a case in which the work on constructible sets has applications to analysis; it allows us to characterize various subsets of $ \R $ as expressible in terms of finite unions, finite intersections and complements of open sets. In some sense, this proof shows that the topological properties of $ \Q $ forbid it to be written in a finite sense using intersections, unions or complements of open or closed sets. In order to facilitate this proof, we begin with a few definitions and lemmas.

\begin{notation}
  Throughout this section, we will consider $X=\R$ and \su the usual topology, that is, arbitrary unions of intervals of the form$(a,b)$ in $\R$. Hence when we say open set or closed set, we are referring to the usual topology. We will also use the letter $F$ to represent a closed set, and the letter $G$ to represent an open set.
\end{notation}

Even though the proof of the following lemma is somewhat simple, we will prove it as this technique will be used in two more results.
\begin{lemma}\label{lmmswitchoc}
   Let $H$ be a subset of $\R ,$ \ $G_1,...,G_n$ open subsets of $\R$ and $F_1,...,F_n$ closed subsets of $\R$. If $H = \bigcup_{k = 1}^{n}\left(G_{k} \cap F_{k}\right)$  then there exist $\hat{G}_{1},...,\hat{G}_{m}$ open sets and $\hat{F}_{1},...,\hat{F}_{m}$ closed sets such that
   \begin{align*}
      H = \bigcap\limits_{j = 1}^{m} \left(\hat{G}_{j} \cup \hat{F}_{j}\right).
   \end{align*}
   Conversely, if $H = \bigcap_{k = 1}^{n}\left(G_{k} \cup F_{k}\right)$, then there exist  $\hat{G}_{1},...,\hat{G}_{m}$ open sets and $\hat{F}_{1},...,\hat{F}_{m}$ closed sets such that
   \begin{align*}
      H = \bigcup\limits_{j = 1}^{m} \left(\hat{G}_{j} \cap \hat{F}_{j}\right).
   \end{align*}
\end{lemma}

\begin{proof}
   We will only prove the first part by induction on $n$, the other part is analogous. \ When $n = 1$, it clear that
   \begin{align*}
      H &= G_{1} \cap F_{1}  = \bigcap\limits_{j = 1}^{2} \left(\hat{G}_{j} \cup \hat{F}_{j}\right)
   \end{align*}
   where $\hat{G}_{1} = G_{1},\ \hat{F}_{1} = \emptyset,\ \hat{G}_{2} = \emptyset$ and $\hat{F}_{2} = F_{1}$. Assume now that the results holds for $n=k$ and consider the case $n=k+1.$ Using our induction hypothesis, we conclude that there exists $\hat{G}_{1},...,\hat{G}_{m}$ open sets and $\hat{F}_{1},...,\hat{F}_{m}$ closed sets such that
   \begin{align*}
      H &= \bigcup\limits_{k = 1}^{n+1} \left(G_{k} \cap F_{k}\right)  = \left[\bigcap\limits_{j = 1}^{m}
          \left(\hat{G}_{j} \cup \hat{F}_{j}\right) \right] \cup \left(G_{n+1} \cap F_{n+1}\right) \\
        &= \left[\left(\bigcap\limits_{j = 1}^{m} \left(\hat{G}_{j} \cup \hat{F}_{j}\right)\right)
           \cup G_{n+1} \right] \cap \left[\left(\bigcap\limits_{j = 1}^{m} \left(\hat{G}_{j} \cup \hat{F}_{j}\right)\right) \cup F_{n+1} \right]  \\
        &=\left(\bigcap_{j = 1}^{m} \left[\left(\hat{G}_{j} \cup G_{n+1}\right)\cup \hat{F}_{j}
           \right] \right) \ \cap \ \left(\bigcap_{j = 1}^{m} \left[\hat{G}_{j} \cup \left(\hat{F}_{j}\cup F_{n+1}\right)\right]\right)
   \end{align*}
    which is overall an intersection of sets of the form $G\cup F$ where $G$ is an open set and $F$ is a closed set, as we wished.
\end{proof}

\begin{lemma}\label{lmmuinftopology}
   Let $n\ge 0$ be given and let $H \in \mathcal{U}_{n}.$  Then there exist $G_{1},...,G_{N}$ open sets and $F_{1},...,F_{N}$ closed sets such that
   \[H = \bigcup\limits_{j = 1}^{N} \left(G_{j} \cap F_{j}\right)\]
\end{lemma}

The proof of this lemma is done by induction and it is an immediate consequence of the previous one. We recall the following topological definition found in~\cite{RoydenFitzpatrick2010} where $\overline{A}$ represents the closure of the set $A$.

\begin{definition}
We define a set $A \sub  \R $ to be nowhere dense if $\overline{A}$ has empty interior. That is, $A$ is nowhere dense if its closure contains no nonempty open intervals.
\end{definition}

 These type of sets help us to define \textit{first category} and \textit{second category} sets which are important pieces of \textit{Baire's Theorem}. A full discussion and examples can be found in ~\cite{RoydenFitzpatrick2010}, Chapter 7, Section 8. Here we have a simple implication of such a theorem that will be useful to our results.

\begin{lemma}\label{lmmweakenedbaire}
   Let $A_{1},A_{2},...,A_{n}$ be nowhere dense sets in $ \R $. Then  \,$\bigcup_{k = 1}^{n} \overline{A_{k}} \neq \R $.
\end{lemma}

\begin{proof}
  This is an immediate consequence of Baire's Theorem as $A_n$ being nowhere dense makes $\overline{A_n}$ nowhere dense too and hence $\R$ would be a set of the first category, a contradiction according to Baire's Theorem.
\end{proof}

\begin{theorem}\label{thmqnotconstructible}
  The set of rational numbers is not constructible from $\u $. That is, $ \Q  \notin \mathcal{U}_{\infty}$.
\end{theorem}

\begin{proof}
   Assume otherwise. Since $\Q \in \ui,$  there is $n\ge 1$ such that $\Q\in\un.$ By Lemma \ref{lmmuinftopology}, there exists $G_{1},...,G_{N}$ open sets and $F_{1},...,F_{N}$ closed sets such that $\Q  = \bigcup_{k = 1}^{N} \left(G_{k} \cap F_{k}\right)$.
   Since $\Q$ is a dense subset of $\R$ we must have
   \begin{align*}
      \overline{\Q} &= \bigcup_{k = 1}^{N} \overline{\left(G_{k} \cap F_{k}\right)}  = \R.
   \end{align*}
   From Lemma \ref{lmmweakenedbaire}, it follows that one of the members of this union must not be nowhere dense. Therefore, there exists $k\ge 1$ such that $\overline{G_{k} \cap F_{k}}$ has nonempty interior; hence there is a non-empty open interval $(a,b)$ such that $\left(a,b\right) \sub \overline{G_{k} \cap F_{k}}$. This implies that $(a,b)\sub \overline{F_{k}} = F_{k}$ and also $\left(a,b\right) \sub \overline{G_{k}}$. Since $G$ is open in the usual topology, there are $(a_n,b_n)$ such that $G_k=\bigcup\limits_{n = 1}^{\infty} (a_n,b_n)$. Since $(a,b)\sub \overline{\bigcup\limits_{n = 1}^{\infty} (a_n,b_n)},$  there is $n$ with $(a,b)\cap(a_n,b_n)\neq\emptyset.$  Therefore \ $G_{k} \cap (a,b)$ is non-empty. Hence $G_{k} \cap (a,b) \sub G_{k} \cap F_{k} \sub \bigcup\limits_{k = 1}^{n} \left(G_{k} \cap F_{k}\right) =  \Q $. But $\Q$ does not have non-empty open subsets, a contradiction. We therefore conclude that $\Q\notin\ui$.
\end{proof}

We can conclude that there is a Borel set~\cite{Billingsley1995}, namely $\Q$, that is not constructible from the usual topology. We might think that there is a direct relationship between sets of first category (see~\cite{RoydenFitzpatrick2010}) and constructible sets, but this is false as $\Q$ is of first category and its complement, the irrational numbers, is of second category and both can not be constructed from \,$\u$. Lemma~\ref{lmmuinftopology} characterizes which sets are constructible.

\begin{theorem}\label{thmdensenotconstructible}
  Let $A\sub\R$ be a set of the second category. If $E$ is a dense subset of $A$ with empty interior then $E \notin\ui,$ and $\sc{E}\notin\ui.$
\end{theorem}

Let's now investigate another initial family of subsets \sv consisting of the family of open intervals. In this case the Cantor set is not constructible, any infinite set with Lebesgue measure zero~\cite{Billingsley1995} is not constructible, also any infinite set with empty interior is not constructible either. In~\cite{Grinblat2010}, Grinblat proves that if there exists a non-measurable set to a given measure, then the a certain algebra does not coincide with the algebra of all subsets of $X.$

\begin{theorem}\label{thmOpenIntervalsConstructibleCharact}
  If \,$\vv=\big\{ (a,b) \ :\ a,b\in\R \cup \{-\infty,+\infty\}\big\}$ then a set  $A\sub\R$ is \sv constructible if and only if $A$ can be written as
    \begin{align}\label{EquationSpecialForm}
       P \cup \bigcup_{i=1}^n (a_i,b_i)
    \end{align}%aqui amigo
  where $n\ge 0, \ -\infty\le a_1<b_1<a_2<\cdots a_n<b_n\le\infty$ and $P$ is a finite subset of $\R$ disjoint from $\bigcup_{i=1}^n (a_i,b_i).$
\end{theorem}

\begin{proof} % Of Theorem \ref{thmOpenIntervalsConstructibleCharact}
   Clearly for any real number $x$ we have that
   $\{x\}=(-\infty,x)^c\cap(x,\infty)^c$, hence $\{x\}$ is constructible (in fact it is 2-constructible). Now if $A$ has the form \ref{EquationSpecialForm}, it can be constructed in at most $n\vee (|P|+2)$ steps. We will prove the converse by induction, more precisely, we will prove that if $A\in\us{m}$ then $A$ has the special form \ref{EquationSpecialForm}.
   \begin{enumerate}
     \item When $m=1$ we have two cases. If $A=(a,b),\, (a,b)\cup (c,d)\, $ or $\, (a,b)\cap (c,d),\ \, A$ can be expressed as the union of disjoint open intervals (including $(a,\infty)$ or $(-\infty,b)$), in which case $A$ has the form \ref{EquationSpecialForm} with $P=\emptyset$.  If $A=\sc{(a,b)}$ with $a,b\in\R$, then $A=\{a,b\}\cup (-\infty,a)\cup(b,\infty)$; if $a=-\infty$, then $A=\sc{(a,b)}=\{b\}\cup (b,\infty)$, and if $b=+\infty$ then $A=\sc{(a,b)}=\{a\}\cup (-\infty,a)$. The case $a=-\infty$ and $b=\infty$ reduces to $A=\R$, hence $A^c=\emptyset$, which certainly has the form \ref{EquationSpecialForm} with $P=\emptyset$ and $n=0$. In all the cases, $A$ has the form \ref{EquationSpecialForm}.
     \item Suppose that statement holds for $m=k.$
     \item Consider $A\in\us{m+1}.$ By induction hypotheses we can assume that $A$ is the union of two elements of the form \ref{EquationSpecialForm} (case 1), the intersection of two elements of such form (case 2) or the complement of an element of such form (case 3). In case 1, it is clear the union of two elements of the form \ref{EquationSpecialForm} has the same form, we just need to rewrite the union of the two disjoint unions as a huge disjoint union and adjust the singletons. In the second case, assume that
         \begin{align}\label{EquationSpecialFormTwo}
             A= \left(P \cup \bigcup_{i=1}^n (a_i,b_i)\right)\cap \left(P' \cup \bigcup_{i=1}^r (a'_i,b'_i)\right),
         \end{align}%aqui amigo
         which is the following union of singletons
         \[ \left[ P \cap \left(P' \cup \bigcup_{i=1}^r (a'_i,b'_i)\right)\right]\cup \left[\left(\bigcup_{i=1}^n (a_i,b_i)\right)\cap  P'\right]
         \]
         and
         \[  \left(\bigcup_{i=1}^n (a_i,b_i)\right)\cap \left(\bigcup_{i=1}^r (a'_i,b'_i)\right).
         \]
     Here we use the fact that the intersection of two finite unions of open intervals is again a disjoint union of open intervals. Making the correct singletons-adjustment we conclude that the set in Equation \ref{EquationSpecialFormTwo} has the form \ref{EquationSpecialForm}. Finally, in case 3
     \begin{align}\label{EquationSpecialFormThree}
             A= \sc{\left(P \cup \bigcup_{i=1}^n (a_i,b_i)\right)}=\sc{P}\cap\bigcap_{i=1}^n \sc{(a_i,b_i)}.
         \end{align}%aqui amigo
   \end{enumerate}
   Here $\sc{P}$ is clearly a union of open intervals (hence it is of the form \ref{EquationSpecialForm}). Let's assume without lost of generality that $-\infty<a_1<b_1<a_2<b_2<\cdots<a_n<b_n<\infty$. Hence
    \[\bigcap_{i=1}^n \sc{(a_i,b_i)}= (-\infty,a_1)\cup (b_1,a_2)\cup (b_2,a_3) \cdots(b_n,\infty)\cup \{a_1,b_1,...,a_n,b_n\},
    \]
    which is of the form \ref{EquationSpecialForm}. Therefore $A$ is the intersection of two sets of the form \ref{EquationSpecialForm} and by case 2, $A$ has the form \ref{EquationSpecialForm}.
\end{proof}

Theorem~\ref{thmOpenIntervalsConstructibleCharact} characterizes the elements of the algebra generated by \,$\vv$. The following theorem, found in~\cite{Folland1999} (Prop. 1.19, p.36), characterizes the elements of the sigma algebra generated by \sv. The similarities are remarkable.

\begin{theorem}\label{thmOpenSigmaAlgebra}
  If \,$\vv=\big\{ (a,b) \ :\ a,b\in\R \cup \{-\infty,+\infty\}\big\}$ then a set  $A\sub\R$ is in the sigma algebra generated by \sv  if and only if $A$ can be written as
    \begin{align*}
       P \cup \bigcup_{i=1}^\infty F_i
    \end{align*}%aqui amigo
  where $F_i$ are closed sets and $P$ is a set of Lebesgue measure zero.
\end{theorem}

If we do not allow infinite intervals in our original family we could prove a similar version of Theorem~\ref{thmOpenIntervalsConstructibleCharact} in the same manner.

\begin{theorem}\label{thmBoundedOpenIntervalsConstructibleCharact}
  If \,$\vv=\{ (a,b) \ :\ a,b\in\R \}$ then a set  $A\sub\R$ is \sv constructible if and only if $A$ can be written as
    \begin{align*}
       P \cup \bigcup_{i=1}^n (a_i,b_i)
    \end{align*}%aqui amigo
  where $a_1,b_n$ are both finite or both infinite, \  $a_1<b_1<a_2<\cdots<b_n$ and $P$ is a finite set disjoint from $\bigcup_{i=1}^n (a_i,b_i).$
\end{theorem}

We finalize this section with an observation on specific constructible sets.

\begin{observation}
  If \,$\vv=\{ (-\infty,a),(b,\infty) \, :\, a,b\in\R \}\}$ then any type of interval  is 2-constructible. If \sv consists of the open sets then the Cantor set is 1-constructible. If \sv consists of the open intervals then for any $n\ge 1$ the Cantor set is not $n-$constructible.
\end{observation}

We will next focus on finite $X$ in order to find ways to characterize sets that are in fact constructible.

%%%%%%%%%%%%%%%%%%%%%%%%%%%%%%%%%%%%%%%%%%%%%%%%%%%%%%%%%%%%%%%%%%%%%%%%%%%%%%%%%%%%%%%%%%%%%%%%%%%%%%%%%%%%%%%%%%%%%%%%%%%%%%%%%%%%%%%%%%%%%%%%%%%%%%%%%%%%%%%%%%%%%%%%%%%%%%%%%%%%%%%%%%%%%%%%%%%%%%%%%%%%%%%%%%%%%%%%%%%%%%%%%%%%%%%%%%%%%%%%%%%%%%%%%%%%%%%%%%%%%%%%%%%%%%%%%%%%%%%%%%%%%%%%%%%%%%%%%%%%%%%%%%%%%%%%%%%%%%%%%%%%%%

\section{Separability and Partitions} %%%%%%%%%%%  555555555555555555555555555
\label{SectionSeparabilityAndPartitons}

\indent Through this work, it has become apparent that an examination of the forms of the sets in an initial family can give information regarding the sets in the last family. In the examples presented from finite initial families, it is worth to note that when constructing $\uo,\uoo,..\un,..$ some elements of $X$ are always grouped together; that is, there are groups of elements in $X$ that either appear all together in each element of $\un$, or none appear.

\begin{example}
   Consider  $X = \{1,2,3,4\}$ and the initial family $\u  = \{\{1,2\},\{3\}\}$. We then have
   \begin{align}
      \u  &= \{\{1,2\},\{3\}\}, \nonumber \\
      \uo  &= \u  \cup \{\{1,2,3\},\emptyset,\{3,4\},\{1,2,4\}\}, \nonumber \\
      \uoo  &= \uo  \cup \{X, \{4\}\}, \nonumber \\
      \mathcal{U}_{\infty} = \uooo  &= \uoo.  \nonumber
   \end{align}
   Thus, it is clear that the elements of the set $\{1,2\}$ are always paired; they never appear separately.
\end{example}

Examples such as this have motivated the following definition.

\begin{definition}\label{defnseparable}
  Let \su be a family of subsets of  $X$. Two elements $a, b \in X$ are \emph{separable} in \su if and only if there exists a set $A \in \mathcal{U}$ such that either $a \in A, \,b \notin A$ or $a \notin A, \,b \in A$. If any two different elements are separable in \suu, we say that $X$ is separable in \suu. If two elements are not separable in $\mathcal{U}$, we say that they are unseparable in $\mathcal{U}$.
\end{definition}

When the family \su is understood and $X$ is separable in \suu, we will simply say that $X$ is \textit{separable}. This definition is very useful through the whole paper. If our initial family is a topology $\tau$ and $X$ is separable in $\tau$ then $X$ satisfies the first axiom of separability ($T_0$) in the topological sense~\cite{RoydenFitzpatrick2010}. If the initial family $\u$ is separable, then $\mathcal{U}_{\infty}$ is precisely the full power set, as noted in the following theorem:

\begin{theorem}\label{TheoremXSeparableUiPower}
  If \su is given and $X$ is finite and separable then \,$\ui=2^X$, that is,  \,$\ui$ consists of all possible subsets of $X.$
\end{theorem}

\begin{proof}
   Clearly, \,$\ui$\, is finite. Let $a\in X$ and consider $f(a)=\min\{|A| \, :\, A\in\ui, \, a\in A\}$. We claim that $f(a)=1.$  Suppose that $f(a)>1$ and take $A$ a set where this minimum is achieved, that is,  $A\in\ui$ with $a\in A$ and $f(a)=|A|.$ Since $|A|>1$, there exist $b\in A$ such that $b\neq a.$ Since $X$ is separable in \,$\u$\, there exists $B\in\u$ such that $a\in B, \, b\in \sc{B}$ or $a\in \sc{B}, \, b\in B.$  Define
    \[
      \define{C}
                 { A\cap B }
                 { a\in B, \, b\in \sc{B} }
                 { A\cap \sc{B} }
                 { a\in \sc{B}, \, b\in B }
    \]
   Then clearly $C\in\ui, \, a\in C$ and $|C|<|A|$ as $b\in A-C,$ contradicting the selection  of $A.$ Since $f(a)=1$ and $a$ was arbitrary, we conclude that \,$\ui$\, contains all singletons, hence it contains all possible unions of such and therefore  \,$\ui$ consists of all possible subsets of $X.$
\end{proof}

We immediately state the  following lemma which will help us to prove some future theorems.

\begin{lemma}\label{LemmaEquivalenceRelationship}
  If \su is given and we define $a\sim b$ in $X$ if and only if $a$ and $b$ are unseparable, then the relationship $\sim$ is an equivalence relationship.
\end{lemma}

\begin{proof}
  Clearly $a\sim a$ as there is not $A\in\u$ such that $a\in A$ and $a\notin A.$ From the definition we can see that if $a\sim b$ then $b\sim a$. Now suppose that $a\sim b, \, b\sim c$ and suppose that $a\not\sim c$ hence $a,c$ are separable. Then there is $A\in\u$ such that either $a\in A, \, c\notin A$ or $a\notin A, \, c\in A$. If $a\in A, \, c\notin A$, since $a$ and $b$ are unseparable,   $a,b\in A$; since $b$ and $c$ are also unseparable,  $b,c\in A$; Therefore $a, b, c\in A,$ a contradiction. The other case $a\notin A, \, c\in A$ is treated in the same manner.
\end{proof}

From now on, we will denote by $[a]_\u$ (or simply $[a]$) the equivalence class of $a$ under this relationship. It is also possible, given an algebra, to look backwards at the generating families based on separable elements. We will in fact be able to correspond every algebra with a particular classification of elements based on separability.

\begin{example}\label{ExampleAlgebraPartition}
Consider the universe $X = \{1,2,3,4,5\}$ and the algebra $\mathcal{A} = \{\emptyset,\{1,2\},\{3,4\},\{5\},\{1,2,3,4\},\{1,2,5\}, \{3,4,5\},X\}$. It follows that there are three different groupings of elements: $1$ and $2$ are unseparable, $3$ and $4$ are unseparable, but $5$ is separable from all other elements. The sets in \sa consist of all possible combinations of the sets $\{1,2\}$, $\{3,4\}$, and $\{5\}$. Note further that the sets $\{1,2\},\{3,4\},\{5\}$ form a partition of the universe $X$.
\end{example}

Examples such as this have led to much more general results on the relationship between partitions and finite algebras. As this example shows, it is possible to work backwards from an algebra and derive a small initial family which generates the algebra. These ideas have led to the following three results.

\begin{lemma}\label{LemmaCharactizeElementsAlgebraPartition}
   Let  $\u  = \{A_{1},\mathellipsis, A_{n}\}$ be a finite partition of $X$. Then $\left|\mathcal{U}_{\infty}\right| = 2^{n}$. Further, $A \in \mathcal{U}_{\infty}$ if and only if there exists a set $\Lambda \sub \{1,\mathellipsis,n\}$ such that
   \begin{align}\label{EquationUnionPartition}
      A = \bigcup\limits_{\lambda \in \Lambda} A_{\lambda}.
   \end{align}
   Also  $a,b$ are unseparable in \,$\ui$ if and only if $a,b$ are unseparable in \,$\u.$
\end{lemma}

\begin{proof}
   Clearly, if  $A$  has the form as in Equation~\ref{EquationUnionPartition}, then it can be constructed by unions and is thus an element of \,$\ui$.
   % Note that, since the family $\u $ forms a partition of $X$, each pair of elements in $A_{i}$ are unseparable in \,$\u$;
   % further, any element of $A_{i}$ is separable from any element of $A_{j}$ in \,$\u$ provided $i \neq j$.
   % Consider now an arbitrary $A \in \mathcal{U}_{\infty}$; form then define the set $\Lambda \sub \{1,\mathellipsis,n\}$ by
   %
   % \begin{align}\label{DefinitionLambda}
   %  \Lambda = \{ k\le n \,|\, A \cap A_{k} \neq \emptyset \}
   % \end{align}
   By Lemma~\ref{LemmaUiPreserved}, $\ui$ is an algebra and a sigma algebra. It is well known (\cite{Billingsley1995}, P.2.21) that if $A\in\ui$ then $A$ is a  union of some elements in the partition, hence $A$ satisfies Equation~\ref{EquationUnionPartition} for some $\Lambda$ subset of $\{1,2,...,n\}.$
   To find $|\ui|$, note that every set $A \in \mathcal{U}_{\infty}$ can be written as unions of $k$ sets in $\u $ and $k \in \{0,1,...,n\}$. The number of ways to select such possible sets is given by $ \binom{n}{0} + \binom{n}{1} + \mathellipsis \binom{n}{n} = 2^{n}.$

   Finally, suppose that $a,b$ are unseparable in \suu. If $a,b$ were separable in \,$\ui$, without loss of generality, there would exists a set $A$ in \,$\ui$ such that $a\in A$ and $b\notin A.$ Since $A$ satisfies Equation~\ref{EquationUnionPartition} for some $\Lambda$ subset of $\{1,2,...,n\},$  there would exist only one $k\in\Lambda$ such that $a\in A_k, b\notin A_k,$ contradicting the fact that $a,b$ are unseparable in \,$\u.$ Conversely, if $a,b$ are separable in \suu, since \,$\u\sub\ui$, by definition of separability, $a,b$ are necessarily separable in \,$\ui,$ the converse of the last statement finishes the proof of this lemma.
\end{proof}

When $X=\{1,2,...,n\}$ one might think that in Lemma~\ref{LemmaCharactizeElementsAlgebraPartition} the only choice for \su that provides \,$\ui=2^X$ is the family of all singletons; however, this is only one case since the family \su in Theorem~\ref{TheoremXSeparableUiPower} does not need to be a partition of $X,$ as in the example $X=\{1,2,3\}, \,\u=\{\{1\},\{2\}\}.$ Note also that in Lemma~\ref{LemmaCharactizeElementsAlgebraPartition}, $X$ might be infinite.

\begin{lemma}\label{LemmaAlgebraClassEquivalence}
   Let $X$ be an arbitrary set and \sa an  algebra of subsets of $X.$  Consider $\mathcal{P}$ the partition induced by the equivalence relationship given in Lemma~\ref{LemmaEquivalenceRelationship}.  Then
   \begin{align}\label{EquationFormClassEquivalence}
       [a]= \bigcap_{\substack{A\in\a\\ a\in A}} A.
   \end{align}
   Equation~\ref{EquationFormClassEquivalence} does not necessary hold if $\a$ is not an algebra.
\end{lemma}

\begin{proof}
    Take $b\in [a]$ and $A\in\a$  such that $a\in A.$ If $b\notin A$, then $a$ and $b$ are separable, contradicting the fact that $a\sim b,$ therefore $b\in A$ and hence necessarily $b$ belongs to the right hand side of Equation~\ref{EquationFormClassEquivalence}. Conversely, take an element $b$ in the right hand side of Equation~\ref{EquationFormClassEquivalence}. If $b$ and $a$ were separable then there would exist a set $H\in\a$ such that either $a\in H, \, b\notin H$ or $a\notin H, \, b\in H.$  In either case, by considering the possibility $\sc{H}\in\a$, we conclude that there exists a set $A\in\a$ such that $a\in A, \, b\notin A,$ contradicting the fact that $b\in\bigcap_{A\in\a, a\in A} A.$ Therefore necessarily $b$ and $a$ are unseparable, hence $b\in [a].$ \, For an example where Equation~\ref{EquationFormClassEquivalence} does not necessary hold, consider $X=\{1,2,3,4\}, \,\u=\{\{1\},\{1,2\},\{1,2,3,4\}\}$. The equivalence classes are $[1]=\{1\}, \, [2]=\{2\}$ and $ [3]=[4]=\{3,4\}.$ In this example $\bigcap_{A\in\u\, :\, 3\in A} A =\{1,2,3,4\}\neq[3].$
\end{proof}

\begin{lemma}\label{LemmaAlgebraPartitionAlgebra}
   Let  \,$\a$ be an algebra of subsets of $X.$  Consider \,$\p$ the partition induced by the equivalence relationship given in Lemma~\ref{LemmaEquivalenceRelationship}. If \,$\p$ is finite then \sa is finite and $\p_\infty=\a.$
\end{lemma}

\begin{proof}
   Consider $\p=\{A_1,...,A_n\}$ be the partition induced by the relationship in Lemma~\ref{LemmaEquivalenceRelationship} and take $E\in\a.$  Clearly
   \[E=\bigcup_{k=1}^n (E \cap A_k) = \bigcup_{\substack{k\le n\\ E \cap A_k\neq\emptyset}} (E \cap A_k).
   \]\
   Take now $k\le n$ such that $E \cap A_k\neq\emptyset.$ Consider any $a\in A_k.$ Since $E \cap A_k\neq\emptyset$, there exists $b\in E \cap A_k.$ Since $b\in A_k, \, a\in A_k$ and $A_k$ is an equivalence class, we conclude that $a\sim b.$ By Lemma~\ref{LemmaAlgebraClassEquivalence} we have
   \begin{align}\label{EquationEqualClassesEquivalence}
       [a]= \bigcap_{\substack{A\in\a\\ a\in A}} A \ \, = \ \, [b] = \bigcap_{\substack{B\in\a\\ b\in B}} B.
   \end{align}
   Since $a\in[a],$  Equation~\ref{EquationEqualClassesEquivalence} implies that $a\in \bigcap_{B\in\a, b\in B} B.$  Therefore $a\in E$ as $E\in \a$ and $b\in E.$ \, Since $a$ was arbitrary in $A_k$ we have that $A_k\sub E\cap A_k.$ Since the other containment is always true, we conclude that $A_k= E\cap A_k$, therefore
   \[ E= \bigcup_{\substack{k\le n\\ E \cap A_k\neq\emptyset}} (E \cap A_k) = \bigcup_{\substack{k\le n\\ E \cap A_k\neq\emptyset}} A_k.
   \]
   Given that each $A_k\in\p$, we conclude that $E\in\p_\infty,$ hence \,$\a\sub\p_\infty.$ By Lemma~\ref{LemmaCharactizeElementsAlgebraPartition}, \,$\p_\infty$ is finite, therefore \sa is finite.
   Conversely, take $A_k\in\p.$ Since \sa is finite and $A_k$ is an equivalence class, Lemma~\ref{LemmaAlgebraClassEquivalence} implies that $A_k$ is a finite union of elements of \,$\a$, thus $A_k\in\a$, therefore \,$\p\sub\a.$ By Lemma~\ref{LemmaAlgebraContains}, $\p_\infty\sub\a.$   Hence $\p_\infty=\a.$
\end{proof}
% Open problem, if $\p$ is infinite, necessarily $\a$ is infinite.
% Not necessarily $\a=\p_\infty$ but % $\p_\infty$ is the smallest sub-algebra $\u$ of $\a$
% such that any pair of elements is unseparable in $\a$ iff they are unseparable in $\u.$

\begin{example}
From Example~\ref{ExampleAlgebraPartition}, when $X = \{1,2,3,4,5\}$ and $\mathcal{A}$ is the algebra $\{\emptyset,\{1,2\},\{3,4\},\{5\},\{1,2,3,4\},\{1,2,5\}, \{3,4,5\},X\},\ \ 1$ and $2$ are unseparable, $3$ and $4$ are unseparable, but $5$ is separable from all other elements, hence the equivalence classes are \, $[1]=\{1,2\}$, \, $[3]=\{3,4\}$ and $[5]=\{5\}.$ Hence the partition induced by the equivalence classes is $\p=\Big\{\{1,2\},\{3,4\},\{5\}\Big\}.$ By Lemma~\ref{LemmaAlgebraPartitionAlgebra}, \, $\p_\infty=\mathcal{A}$. Indeed,\\
   \begin{align}
      \p  &= \Big\{\{1,2\},\{3,4\},\{5\}\Big\}, \nonumber \\
      \p_1  &= \p  \cup \Big\{\{1,2,3,4\},\,\{1,2,5\},\emptyset,\{3,4,5\}\Big\}, \nonumber \\
      \p_2  &= \p_1  \cup \{X\}, \nonumber \\
      \p_\infty &= \{\emptyset,\{1,2\},\{3,4\},\{5\},\{1,2,3,4\},\{1,2,5\}, \{3,4,5\},X\}  \nonumber \\
      \p_\infty &= \mathcal{A}.  \nonumber
   \end{align}

\end{example}

\begin{theorem}\label{TheoremPartitionsAlgebras}
   Let $X$ be a finite universe. Let $\mathbf{A}$ be the collection of nonempty algebras of subsets of $X$, and $\mathbf{P}$ the collection of partitions of the set $X$. Then each element in $\mathbf{A}$ corresponds to a unique element in $\mathbf{P}$, i.e.  there exists a bijective mapping $f: \mathbf{A} \longrightarrow \mathbf{P}$.
\end{theorem}

\begin{proof}
   Let $\a \in \mathbf{A}$ be an arbitrary non-empty algebra; as previously, we define the relationship on $\a$ by $a \sim b$ if and only if $a$ and $b$ are unseparable in \,$\a$. By Lemma~\ref{LemmaEquivalenceRelationship}, this becomes an equivalence relationship in $X$, hence it induces a partition in $X$. Denote this partition by $\p_{\a}$. Define the function $f: \mathbf{A} \longrightarrow \mathbf{P}$ by
   \[
   f\left(\a\right) = \p_{\a}
   \]
   %We claim that $f$ is a bijection.\\
   \indent Consider an arbitrary partition \,$\p = \{A_{1}, \mathellipsis, A_{n}\}$ of the universe $X$. Form the family \,$\u = \{A_{1},\mathellipsis,A_{n}\}$ and consider the algebra \,$\ui$.  By Lemma~\ref{LemmaCharactizeElementsAlgebraPartition}, if a pair of elements $a,b$ are unseparable in \,$\ui$ then they are unseparable in \,$\u.$ Therefore the classes of equivalence given by this relationship are precisely the elements of the partition \,$\p.$ The way we are defining our function, we conclude that  $f\left(\u_{\infty}\right) = \p.$ Therefore $f$ is surjective.  Take \,$\a, \,\b\in\mathbf{A}$ with \,$f(\a)=\p_\a, \ \,f(\b)=\p_\b$ and \,$\p_\a=\p_\b.$  By using  Lemma~\ref{LemmaAlgebraPartitionAlgebra}, we conclude that
   \begin{align}
     \a = C_{\infty}\left(\p_\a\right) = C_{\infty}\left(\p_\b\right) = \b.\nonumber
   \end{align}
   Hence, $f$ is injective.
\end{proof}

For another proof of this theorem and a further discussion of these connections, see the upcoming paper~\cite{GaBaMoBoPoCe}.
This theorem has also led to a simple method of determining how many distinct algebras exist for a given universe. If we recall that the $n^{th}$ Bell number represents the total number of partitions of the set $\{1,2,...,n\}$ ~\cite{Comtet1994}, the following corollary is an immediate consequences of Theorem~\ref{TheoremPartitionsAlgebras}.

\begin{corollary}\label{corcountalg}
Let $X$ be a universe of order $n \in \mathds{N}$. Then there exist exactly $B_{n}$ distinct non-empty algebras of subsets of $X$, where $B_{n}$ is the $n^{th}$ Bell number.
\end{corollary}

This corollary shows that, although there can be an extremely large number of families of subset of $X$, they only generate a small collection of algebras. Of the $65,536=2^{2^{4}}$ families of subsets of $\{1,2,3,4\}$, there are only $B_4=53$ distinct algebras generated. In general, the connections between partitions of the universe and the algebras have led to several of the major results of this paper.

If we start with an algebra \saa, Lemma~\ref{LemmaAlgebraPartitionAlgebra} provides a partition that generates the original algebra \saa. Suppose now that we start with a general family \su that is not necessarily an algebra. The following lemma provides a partition, based on the initial family \suu, that generates an algebra that contains \,$\u,$ in other words, to achieve a partition that generates an algebra of sets, it is not necessary to look at the algebra of sets but to look at a smaller family that generates such algebra.

 Let \,$\u=\{E_1,E_2,...,E_n\}$ be a family of subsets of $X$.  To shorten notation, let  $I_n=\{1,2,...n\}.$ \, For $I\sub I_n$ define
   \begin{align}\label{EquationAIndices}
       A_I:=\big\{ x\in X \, : \, I=\{i\in I_n\, :\, x\in E_i\}\big\}.
   \end{align}
 $A_I$ consists of the elements $x$ of $X$ such that $I$ is exactly the set of indices that correspond to the  members of \su where $x$ belongs. Note that
 \begin{align}\label{EquationAIUnionMinusIntersections}
    \defi{A_I}
         {\ds{\bigcap_{i\in I} E_i-\bigcup_{i\notin I} E_i} & : \, \mathrm{\ if\ } I\neq\emptyset,\, I\neq I_n  \\[6mm]
         \ds{\bigcap_{i=1}^n E_i} & : \, \mathrm{\ if\ } I=I_n,  \\[3mm]
         \ds{X-\bigcup_{i=1}^n E_i} & : \, \mathrm{\ if\ } I=\emptyset.}
 \end{align}

\begin{theorem}\label{TheoremAnyFamilyPartitionAlgebra}
   Let \,$\u=\{E_1,E_2,...,E_n\}$ be a family of subsets of $X$ and define $A_I$ by Equation~\ref{EquationAIndices}. If $\h=\left\{ A_I\,:\, I\sub I_n\right\}$ then \sh is a partition of $X$ and $\hs{2^{n}}$ is an algebra that contains \suu.
\end{theorem}

\begin{proof}
   Take $k\le n$. We will prove that
   \begin{align}\label{EquationSetEIndices}
      E_k = \bigcup_{\substack{I\sub I_n \\ k\in I\,\,}} A_I.
   \end{align}
   Take $x_0\in E_k.$  If we define $I_0=\{i\in I_n\,:\, x_0\in E_i\},$ then $k\in I_0$ and $x_0\in A_{I_0}.$ Hence $x_0\in \bigcup_{I\sub I_n\,:\, k\in I} A_I.$ and since $x_0$ was arbitrary in $E_k$, we conclude that  $E_k\sub \bigcup_{I\sub I_n\,:\, k\in I} A_I.$
   Conversely, take $x_0\in \bigcup_{I\sub I_n\,:\, k\in I} A_I.$ Then there is $I_0\sub I_n$ with $k\in I_0$ and such that $x_0\in A_{I_0}.$ Therefore $x_0\in E_k,$ this implies that $\bigcup_{I\sub I_n\,:\, k\in I} A_I \sub E_k.$  Therefore Equation~\ref{EquationSetEIndices} holds.\\[1mm]
   \indent Equation~\ref{EquationSetEIndices} tells us that $E_k$ can be constructed from \,$\h$ in at most $m$ steps where $m$ is the number of subsets of $I_n$ that contain $k$; this number does not exceed $2^{n}.$ Hence \,$\u\sub\hs{2^{n}}.$ \, Note that \sh has at most $2^n$ elements.

   It is clear that if $x_0\in A_I\cap A_J,$ then $I=\{i\in I_n\, :\, x_0\in E_i\}$ and $J=\{i\in I_n\, :\, x_0\in E_i\},$ hence $I=J.$ Therefore \,$\h$ is a  family of disjoint sets that actually form a partition of $X$, since if $x\in X$ and $I_x=\{k\in I_n\,:\, x\in E_k\}$ then $x\in A_{I_x}.$ By Lemma~\ref{LemmaCharactizeElementsAlgebraPartition}, any element in \,$\hi$ can be expressed as a union of at most $2^n$ elements in \shh. Therefore  \,$\hi\sub\h_{2^{n}}\sub\hi$,  and thus $\hs{2^{n}}$ is an algebra.
\end{proof}

We will see later that for the family \sh that appears in Theorem~\ref{TheoremAnyFamilyPartitionAlgebra}, in fact $\hn$ is an algebra.
In addition to having deep connections with algebras, there are a number of desirable properties that partitions posses. These include easy characterizations of the sets in the \textit{last family}, as well as a determination of the number of steps required to construct a given family. We will analyze these two points.

\begin{definition}
   For any two families \,$\u,\h$ of subsets of $X$, we define
   \[
      S_{\u}\left(\h\right) = \inf\left\{n \ge 0 \, : \,  \h \sub \un\right\},
   \]
   understanding that if for any $n\ge 0, \,\h \not\sub \un$\, then $S_{\u}\left(\h\right) =\infty.$
\end{definition}

Thus, $S_{\u}\left(\h\right)$ represents the number of steps required to construct the family \sh from the family \suu. If $S_{\u}\left(\h\right)=k$, this means that $ \h \sub \u_k$ but $ \h \not\sub \us{k-1}.$ Typically, we would like to evaluate $S_\u(\ui).$

\begin{definition}
   Let \,$\u = \{A_{1},A_{2},...,A_{n}\}$ be a finite family of subsets of a universe $X$. For any $k$ and $m$ with $0\le k, m\le n,$ define
   \begin{align}\label{kUnions}
      U_k\left(\u\right) = \left\{ \bigcup_{i\in I} A_i \, : \, I\sub I_n, \,\, |I|=k\right\}
   \end{align}
   \begin{align}\label{EquationSetBm}
      B_m = \big\{ i\in \{0,1...,n\} \, : \, U_i\left(\u\right) \sub \us{m} \big\}.
   \end{align}
   In other words, $U_k\left(\u\right)$ represents all possible unions of $k$ elements in \suu. Any element in $U_k\left(\u\right)$ will be called a \textit{k-union}.
\end{definition}
Note that $k\in B_{m}$ if and only if $\us{m}$ contains all possible $k-$unions. For the particular case when $m=0, \ k\in B_0$ if and only if \,$\u$ contains all possible $k-$unions. Note also that $U_0\left(\u\right) = \left\{\emptyset\right\}$,  thus, $0 \in B_{m}$ if and only if $\emptyset \in \mathcal{U}_{m}$.

\begin{example}
   Consider $\u  = \{A_{1},A_{2},\mathellipsis,A_{10}\}$ a partition of $X$.  Then
   \begin{align*}
                      B_0 &= \{1\}                   & &\mini{.65}{as \,$\u_0$ contains all 1-unions (singles),}
       \\ \cline{2-4} B_1 &= \{0, 1, 2, 9\}          & &\mini{.65}{as \,$\uo$ contains $\emptyset$, all 1-unions, 2-unions and all 9-unions (complements of 1-unions),}
       \\ \cline{2-4} B_2 &= \{0, 1, 2, 3, 4, 8, 9, 10\} & &\mini{.65}{as \,$\uoo$ contains $\emptyset$, all 1-unions, 2-unions, 3-unions, 4-unions, , 8-unions (complements of 2-unions), and alll 9-unions as well as the 10-union.}
       \\ \cline{2-4} B_3 &= \{0, 1, ..., 10\}        & &\mini{.65}{ as \,$\uooo$ contains all possible $k$-unions, $k=0,1,2,...,10$,}
       \\ \cline{2-4} B_4 &= B_3.  & &
   \end{align*}
\end{example}

\begin{theorem}\label{TheoremMinimumNumberOfSteps}
   Consider $\u = \{A_{1},A_{2},...,A_{n}\}$ a partition of $X$. If $n \ge 4$, then
   \begin{align}\label{FormulaMinimumNumberOfSteps}
      \defi{S_{\u}(\ui)}{\ds{\left\lfloor\log_2 n \right\rfloor } & : \, \mathrm{\ if\ \,} {n \leq 1 + 3\cdot2^{\left\lfloor \log_2 n \right\rfloor - 1}},  \\[3mm]
                         \ds{\left\lceil\log_2 n\right\rceil } & : \, \mathrm{\ if\ \, } {n > 1 + 3\cdot2^{\left\lfloor \log_2 n \right\rfloor - 1}}.}
   \end{align}
\end{theorem}

\begin{proof}
   An inductive argument allows us to conclude the following,
   \begin{align}
      B_{0} &= \{1\}  \nonumber \\
      B_{1} &= \{0,1,2,n-1\}  \nonumber \\
      B_{2} &= \{0,1,2,3,4,n-2,n-1,n\}  \nonumber \\
      &\vdots \nonumber \\
      B_{k} &= \{0,1,2,...,2^{k}, n-2^{k-1},\mathellipsis,n\}.  \nonumber
   \end{align}
   To determine if $\{0,1,2,...,n\}\sub B_k,$ we define the gap function
   \begin{align}
      g_{n}\left(k\right) &= \left(n - 2^{k-1}\right) - 2^{k}  = n - 3 \cdot 2^{k-1}. \nonumber
   \end{align}
   Note that when \,$\u=\{A_1,A_2,A_3,A_4\}$ and $k=1$, even though $g_4(1)=n - 3\le 1,$ this does not imply that \,$S_{\u}(\ui)\le 1$ as in that particular case $B_1=\{0,1,2,3\}$. In fact in this case $S_{\u}(\ui)=2,$ the problem is that $n=4\notin B_1$.

   When $k\ge 2$, the gap function helps us establishing the following criteria,
   \begin{align}\label{SLessThanKIfandonlyif}
      \u_{k} = \ui \Leftrightarrow \{0,1,2,\mathellipsis,n\} \sub B_k \Leftrightarrow g_n(k)\le 1 \Leftrightarrow S_{\u}(\ui)\le k.
   \end{align}
   Regardless of the value of $k$, if $g_n(k)>1$ then $S_{\u}(\ui)>k.$ We now establish two important bounds on  $S_{\u }\left(\mathcal{U}_{\infty}\right)$. Assume that $n\ge 5.$ Since
   \begin{align}
      g_{n}\left(\left\lfloor \log_{2} n\right\rfloor - 1\right)  =
       n - 3\cdot 2^{\left\lfloor \log_{2} n \right\rfloor - 2}
      \geq n - \frac{3}{4} \cdot 2^{\log_{2} n}  = \frac{n}{4}
      > 1, \nonumber
   \end{align}
   we conclude that $S_{\u }\left(\mathcal{U}_{\infty}\right) > \left\lfloor \log_{2} n\right\rfloor - 1$. Thus,  $S_{\u }\left(\mathcal{U}_{\infty}\right) \geq \left\lfloor \log_{2} n \right\rfloor$.
   Since $n\ge 5$, the value $k_0:=\left\lceil \log_{2} n \right\rceil$ satisfies $k_0\ge 2$ and also
   \begin{align*}
      g_{n}(k_0)  = n - \frac{3}{2} \cdot 2^{\left\lceil \log_{2} n\right\rceil} \leq n - \frac{3}{2} \cdot 2^{\log_{2} n}  = n - \frac{3n}{2} \le 1.
   \end{align*}
   The criterion given in~\ref{SLessThanKIfandonlyif} implies that $S_{\u }\left(\mathcal{U}_{\infty}\right) \leq k_0=\left\lceil \log_{2} n \right\rceil$. We have therefore established
   \begin{align}\label{TwoBoundsForS}
      \left\lfloor \log_{2} n \right\rfloor \leq S_{\u }\left(\mathcal{U}_{\infty}\right) \leq \left\lceil \log_{2} n \right\rceil.
   \end{align}
   Combining~\ref{TwoBoundsForS} and ~\ref{SLessThanKIfandonlyif} and using the fact that $k_1:=\left\lfloor \log_{2} n \right\rfloor\ge 2$, gives us
   \begin{align}
      S_{\u }\left(\mathcal{U}_{\infty}\right) = k_1  \Leftrightarrow g_{n}\left(k_1\right) \leq 1  \Leftrightarrow n \leq 1 + 3\cdot 2^{\left\lfloor \log_{2} n \right\rfloor - 1}. \nonumber
   \end{align}
   This proves Formula~\ref{FormulaMinimumNumberOfSteps} for $n\ge 5.$  As observed at the beginning of this proof, for the case $n=4,\ S_{\u}(\ui)=2=\left\lfloor \log_{2} n \right\rfloor$, and clearly $4 \leq 1 + 3\cdot 2^{\left\lfloor \log_{2} 4 \right\rfloor - 1},$ this verifies Formula~\ref{FormulaMinimumNumberOfSteps} when $n=4.$
\end{proof}

Computing the values of $S_{\u}(\ui)$ for smaller numbers gives us the formula
\begin{align}\label{FormulaMinimumNumberOfStepsGeneral}
   \defi{S_{\u}(\ui)}{1 & : \, \mathrm{\ if\ \,} n=1,2,\\
                      2 & : \, \mathrm{\ if\ \,} n=3,\\[0mm]
                      \ds{\left\lfloor\log_2 n \right\rfloor } & : \, \mathrm{\ if\ \,} {n\ge 4, \ n \leq 1 + 3\cdot2^{\left\lfloor \log_2 n \right\rfloor - 1}},  \\[1mm]
                      \ds{\left\lceil\log_2 n\right\rceil } & : \, \mathrm{\ if\ \, } {n\ge 4, \ n > 1 + 3\cdot2^{\left\lfloor \log_2 n \right\rfloor - 1}}.}
\end{align}

The first $n\ge 4$ such that $S_{\u}(\ui)=\ds{\left\lceil\log_2 n\right\rceil }$ is $n=14.$  In the upcoming paper~\cite{GaBaMoBoPoCe}, we work with families of disjoint intervals and we identify a corresponding number $S_{\u}(\ui)$ which leads to a slightly different formula than the one in~\ref{FormulaMinimumNumberOfStepsGeneral}. In that paper, we also build minimal families that generate the power set and we also provide an algorithm to do this in the finite case.

\begin{corollary}\label{CorollaryAnyFamilyPartitionAlgebra}
   Let \,$\u=\{E_1,E_2,...,E_n\}$ be a family of subsets of $X$ and define $A_I$ by Equation~\ref{EquationAIndices}. If $\h=\left\{ A_I\,:\, I\sub I_n\right\}$ then \sh is a partition of $X$ and $\hs{n}=\ui$.
\end{corollary}

\begin{proof}
   By Theorem~\ref{TheoremAnyFamilyPartitionAlgebra}, we have that \, $\hs{2^n}$ is an algebra that contains \suu. \, If we apply Theorem~\ref{TheoremMinimumNumberOfSteps} to the partition $\h$ (which has at most $2^n$ elements) we conclude that $S_{\h}(\hi)\le \left\lceil\log_2 2^n\right\rceil = n,$ which means that \,$\hi\sub \hn,$\, hence \,$\hi=\hn=\hs{2^n},$\, therefore \,$\u\sub\hn.$ Since $\hn$ is an algebra, we conclude  \,$\ui\sub\hn.$ From Equation~\ref{EquationAIUnionMinusIntersections}, we conclude that $\,\h\sub\ui,$ therefore $\hn=\ui.$
\end{proof}

The use of partitions to characterize algebras and vice-versa leads to some very general results. In the next sections, we will connect these ideas with a class of relatively simple families in order to expand the theory.

%%%%%%%%%%%%%%%%%%%%%%%%%%%%%%%%%%%%%%%%%%%%%%%%%%%%%%%%%%%%%%%%%%%%%%%%%%%%%%%%%%%%%%%%%%%%%%%%%%%%%%%%%%%%%%%%%%%%%%%%%%%%%%%%%%%%%%%%%%%%%%%%%%%%%%%%%%%%%%%%%%%%%%%%%%%%%%%%%%%%%%%%%%%%%%%%%%%%%%%%%%%%%%%%%%%%%%%%%%%%%%%%%%%%%%%%%%%%%%%%%%%%%%%%%%%%%%%%%%%%%%%%%%%%%%%%%%%%%%%%%%%%%%%%%%%%%%%%%%%%%%%%%%%%%%%%%%%%%%%%%%%%%%

\section{Minimal Constructibility} %%%%%%%%%%%  666666666666666666666666666
\label{SectionMinimalConstructibility}

As noted before, it is frequently of use to be able to move backwards from complicated algebras to relatively simple generating families. The following example motives the definition that follows it.

\begin{example}
   Consider  $X = \{1,2,3\}$ and the family
   \begin{displaymath}
      \mathcal{U} = \{\{1,2\},\{3\},\emptyset,X\}
   \end{displaymath}
   The elements $\emptyset$ and $X$ are constructible in one step from the other two elements in the family; thus, in some sense, they are unnecessary. The family $\h = \{\{1,2\},\{3\}\}$ is of smaller size, and $\mathcal{U}_{\infty} = \h_{\infty}$. Thus, this allows for the construction of a simpler family that still generates the same algebra.
   % as a more complicated family.
\end{example}

\begin{definition}\label{defn1mc}
   We say \su is \emph{1-minimal constructible} if for any $\h \subsetneq \mathcal{U}$ we have \,$\u \nsubseteq \h_{1}$.
\end{definition}

In other words, a family is 1-minimal constructible if it cannot be built in one step from any of its proper subsets. Thus, these families contain simple structure and lack much of the duplication and unnecessary elements encountered in algebras. It is possible to place a stronger condition on the families, and to require that there are no elements that can be constructed in any given finite number of steps. This desire leads to the following definition.

\begin{definition}\label{defnnmc}
   Let $n \ge 1$. We say that \su is $n$-minimal constructible if for any $\h \subsetneq \mathcal{U}$ we have \,$\u \nsubseteq \h_{n}$.
\end{definition}

\begin{example}
   \label{ex1}
   Consider $X = \left\{ 1, 2, 3, 4 \right\}$, and \,$\u = \left\{ \left\{1 \right\}, \left\{ 2 \right\}, \left\{ 3, 4 \right\} \right\}$. Then
   \begin{displaymath}
      \u_{1} = \left\{ \left\{1 \right\}, \left\{ 2 \right\}, \left\{ 3, 4 \right\}, \left\{1, 2 \right\}, \left\{ 1, 3, 4 \right\}, \left\{ 2, 3, 4 \right\}, \emptyset \right\}.
   \end{displaymath}
   \indent We see that \,$\u$ is 1-minimal constructible because none of $\left\{1 \right\}$, $\left\{ 2 \right\}$ or $\left\{ 3, 4 \right\}$ can be constructed from the remaining two in one step.  Also note that \,$\u$ is not 2-minimal constructible because there is \,$\vv,$\, a proper subfamily of \,$\u,$ such that \,$\u\sub\vv_2$, namely, $\vv=\left\{ \left\{1 \right\}, \left\{ 2 \right\} \right\},$ as
   \begin{displaymath}
      \vv_{2} =  \left\{\emptyset,  \left\{1 \right\}, \left\{ 2 \right\}, \left\{ 1, 2 \right\}, \left\{ 3, 4 \right\}, \left\{ 1, 3, 4 \right\} , \left\{ 2, 3, 4 \right\}, \left\{ 1, 2, 3, 4 \right\} \right\}.
   \end{displaymath}
   \indent On the other hand, $\u_{1}$ is not 1-minimal constructible because $\left\{1, 2 \right\}$ can be constructed in one step through the union of $\left\{1 \right\}$ and $\left\{ 2 \right\}$.
\end{example}
%The previous example shows that we could not say that \su is 2-minimal constructible if and only if \,$\uo$ is %1-minimal constructible, hence a recursive definition for $n$-minimal constructibility is not equivalent.

\begin{example}
   Some examples of 1-minimal constructible families with universe $\mathcal{X} = \left\{ 1, 2, 3 \right\}$ are
   \begin{align}
      \u^{1} &=  \left\{   \right\}\nonumber \\
      \u^{2} &=  \left\{ \left\{ 2  \right\} \right\}\nonumber \\
      \u^{3} &=  \left\{  \left\{ 1 \right\}, \left\{ 1, 3 \right\}  \right\}\nonumber \\
      \u^{4} &= \left\{ \left\{ 3 \right\}, \left\{ 2, 3 \right\}, \left\{ 1, 2, 3 \right\} \right\}\nonumber \\
      \u^{5} &=  \left\{ \left\{ 1 \right\}, \left\{ 2 \right\}, \left\{ 3 \right\},  \left\{ 1, 2, 3 \right\} \right\}\nonumber
   \end{align}
   Note that the families $\u^{1}, \u^{2}, \u^{3}$ are $n$-minimal constructible for every natural number $n$. The families $\u^{4}$ and \,$\u^{5}$ are not 2-minimal constructible.
\end{example}

An alternative characterization of this idea can be formulated. If no set in the family can be constructed in one step from the remaining sets, then no proper subset can construct the entire family. This is formalized in the next lemma.

\begin{lemma}\label{lmmremove1element}
   \su is $n$-minimal constructible if and only if  for any $A\in\u$ we have \,$\u\not\sub C_n(\u-\{A\}).$
\end{lemma}

\begin{proof}
   We will prove the opposite of the statement, that is
   %\exists \,\h \subsetneq \mathcal{U},\ : \,\,\u \sub \h_{n} \,\,
    \[ \u \,\,\mathrm{is\ not\ } n\hy\mathrm{minimal\ constructible\ } \Leftrightarrow \ \, \exists \,A\in\u \, : \,\,\u\sub C_n(\u-\{A\}).
    \]
    It is clear that if there exists $A\in\u$ such that \,$\u\sub C_n(\u-\{A\}),$ then \su would not be $n$-minimal constructible. Conversely, if \su is not $n$-minimal constructible, there exists \sh a proper subfamily of \su such that $\u\sub\h_n.$ Take simply $A\in \u-\h,$ then clearly $\u\sub\h_n\sub C_n(\u-\{A\}).$
\end{proof}

If a family cannot be constructed from any of its subfamilies in $n$ steps, it is clear that it cannot be constructed from any of its subfamilies in less than $n$ steps. Similarly, if a family can be constructed in a given number of steps, it is also constructible in any higher number of steps. This leads to the following lemma whose simple proof we are omitting.

\begin{lemma}\label{lmmmoveback}
   If \su is $n$-minimal constructible and $k \leq n$ then \su is $k$-minimal constructible. If \su is not $n$-minimal constructible and $k \geq n$, then \su is not $k$-minimal constructible.
\end{lemma}

The following lemma, even though simple, will be used in future counting results.

\begin{lemma}\label{lmmdoesnothavex}
   Assume $n\ge 2$ and \,$|\u| > 1.$ \, If \su contains either $\emptyset$ or $X$, then \su is not $n$-minimal constructible.
\end{lemma}

\begin{proof}
   Suppose that $\emptyset \in \u$; since  \,$|\u| > 1$ there must exist some $A \in \u$ with $A \neq \emptyset$. Since $\sc{A} \in \uo$  and $A \cap \sc{A} = \emptyset \in \uoo,$ we conclude \,$\u \sub C_2(\u - \{\emptyset\})$. Lemma~\ref{lmmremove1element} implies that \su is not 2-minimal constructible and  Lemma~\ref{lmmmoveback} tells us that if $n\ge 2$ then \su is not $n$-minimal constructible. A similar argument gives the result in the case that $X \in \u$.
\end{proof}

Given an $n$-minimal constructible family, there are various elementary properties of the family that can be derived. A few of these are summarized in the next lemmas.

\begin{theorem}\label{TheoremMinimalConstructiblePreserved}
   If \su is $n$-minimal constructible and \,$\h\sub\u$ then \,$\h$ is $n$-minimal constructible.
\end{theorem}

\begin{proof}
   Assume the opposite:  then there exists $\mathcal{K} \subsetneq \h$ with $\h \sub \mathcal{K}_{n}$. Then $\u \sub \h \cup \left(\u-\h\right) \sub \mathcal{K}_{n} \cup \left(\u - \h\right)$.
   Lemma~\ref{LemmaBasicTripple} implies that
   \begin{align*}
      \u  \sub \mathcal{K}_{n} \cup \left(\u - \h\right)\nonumber  \sub \mathcal{K}_{n} \cup \left(\u - \h\right)_{n}\nonumber  \sub \left(\mathcal{K} \cup \left(\u - \h\right)\right)_{n}.
   \end{align*}
   If we call \,$\vv= \mathcal{K} \cup \left(\u - \h\right)$ then \,$\vv\sub \h\cup \left(\u - \h\right)\sub\u$. Since $\mathcal{K} \subsetneq \h$, there exists $A\in\h-\k.$  Therefore $A\in\u$; however, $A\notin \k$ and $A\notin \u-\h$ (otherwise $A\notin\h$, which is impossible). Hence $A\notin \vv$, and therefore $\vv\subsetneq \u,$  contradicting the fact that \su is $n$-minimal constructible.
\end{proof}

\begin{corollary}\label{corintersection}
   If \su and $\h$ are $n$-minimal constructible then $\u \cap \h$ is also $n$-minimal constructible.
\end{corollary}

\begin{observation}
   Note that \su is 1-minimal constructible  if and only if for any \,$\h$ a subfamily of \su such that \,$\u_0\sub\ho$\, we have that \,$\h=\u,\, $ (remember that \, $\u_0=\u$).  Hence we could have taken another path to define $n$-minimal constructibility, ``\su is $n$-minimal constructible if and only if for any \,$\h$ subfamily of \su such that \,$\u_{n-1}\sub\hn$\, we have that \,$\h=\u";$ however, this is not equivalent to our definition.
\end{observation}

\begin{definition}\label{defnFatnmc}
   We say \su is n-minimal-fat constructible if for any \,$\h$ subfamily of \su such that \,$\u_{n-1}\sub\hn$\, we have that \,$\h=\u.$
\end{definition}

The following lemma will not be used in this paper; however, it is proven here to explore another generalization of 1-minimal constructible.

\begin{lemma}\label{LemmaMinimalConstructibleOption}
   If \su is $n$-minimal constructible then \su is $n$-minimal-fat constructible. The opposite implication is true when $n=1$ but not always.
\end{lemma}

\begin{proof}
   Let \,$\h\sub\u$ such that $\u_{n-1}\sub\hn.$ Assume that $\h\neq\u.$ Since \su is $n$-minimal constructible and \,$\h\subsetneq\u$, we have that $\u\not\sub\hn$, but $\u\sub\u_{n-1}\sub\hn$ which is a contradiction.

   For the counterexample, consider $X=\{1,2,3,4,5\}$ and \,$\u=\{\{1\},\{2\},\{3\},\{4\},\{2,3,4\}\}.$
   We claim that \su is 2-minimal-fat constructible, to check this, we need to prove that if \sh is a proper subfamily of \su then $\uo\not\sub\h_2.$ It will be enough to check five different cases. \, Note that $\{1,5\}\in\uo.$
   \begin{itemize}
     \item $\h=\{\{1\},\{2\},\{3\},\{4\}\} \subsetneq\u;$ however, \,$\uo\not\sub\h_2$\, since $\{1,5\}\notin\h_2.$
     \item $\h=\{\{1\},\{2\},\{3\},\{2,3,4\}\} \subsetneq\u;$ however, \,$\uo\not\sub\h_2$\, since $\{4\}\notin\h_2.$
     \item $\h=\{\{1\},\{2\},\{4\},\{2,3,4\}\} \subsetneq\u;$ however, \,$\uo\not\sub\h_2$\, since $\{3\}\notin\h_2.$
     \item $\h=\{\{1\},\{3\},\{4\},\{2,3,4\}\} \subsetneq\u;$ however, \,$\uo\not\sub\h_2$\, since $\{2\}\notin\h_2.$
     \item $\h=\{\{2\},\{3\},\{4\},\{2,3,4\}\} \subsetneq\u;$ however, \,$\uo\not\sub\h_2$\, since $\{1\}\notin\h_2.$
   \end{itemize}
   If \,$\h\subsetneq\u$, then \sh is necessary a subfamily of one of the previous five.\\
   We claim that \su is not 2-minimal constructible. Consider the proper subfamily of \su, $\h=\{\{1\},\{2\},\{3\},\{4\}\}.$ \, Note that $\{2,3\},\{4\}\in\ho$\, and hence $\{2,3,4\}=\{2,3\}\cup\{4\}\in\h_2,$\, therefore $\u\sub\h_2.$
\end{proof}

Combining our results, we have proven two general classification theorems that allow the construction of relatively simple generating families for any given finite algebra of sets. These theorems have deep connections to the results on partitions and separable elements. The first theorem gives a method of determining a family with a number of desirable properties relating to elements in $\u_{\infty}$ and the number of steps required to construct the last family; on the other hand, the second theorem gives a construction of a simplest generating family for a given algebra.

\begin{theorem}\label{thmon1mc}
   Let \sa be a finite algebra of subsets of $X$ with \,$|\a|\ge 2^{4}$. Then there exists a family $\h$ of subsets of $X$ such that:
   \begin{itemize}
      \item $\h \sub \a$
      \item The elements of $\h$ form a partition of $X$
      \item $\h$ is 1-minimal constructible
      \item $\h_{\infty} = \a$
      \item $\left|\h\right| = \log_{2} \left|\a\right|$
      \item If $n = \left|\h\right|$, the number of steps required to construct \sa from $\h$ is given by:
             \begin{align}\label{FormulaAHMinimumNumberOfSteps}
                \defi{S_{\h}(\a)}{\ds{\left\lfloor\log_2 n \right\rfloor } & : \, \mathrm{\ if\ \,} {n \leq 1 + 3\cdot2^{\left\lfloor \log_2 n \right\rfloor - 1}},  \\[3mm]
                                   \ds{\left\lceil\log_2 n\right\rceil } & : \, \mathrm{\ if\ \, } {n > 1 + 3\cdot2^{\left\lfloor \log_2 n \right\rfloor - 1}}.}
             \end{align}
   \end{itemize}
\end{theorem}

\begin{proof}
   Consider \,$\h=\p$\, the partition induced by the equivalence relationship given in Lemma~\ref{LemmaEquivalenceRelationship}. By Equation~\ref{EquationFormClassEquivalence} given in Lemma~\ref{LemmaAlgebraClassEquivalence}, any member of the partition is a member of  \sa as it is the union of elements of \saa. Hence necessarily \,$\h\sub\a$\, and then \sh is also finite. Consider $n=|\h|.$  Lemma~\ref{LemmaAlgebraPartitionAlgebra} implies that $\h_{\infty} = \a$ and Lemma~\ref{LemmaCharactizeElementsAlgebraPartition} implies that $|\hi|=2^n.$ Hence $\left|\h\right| = \log_{2} \left|\a\right|$.  Our hypotheses imply that $n\ge 4,$ therefore we can apply Theorem~\ref{TheoremMinimumNumberOfSteps} to the family \,$\u=\h$\,  to find that the number $S_{\u}(\ui)=S_{\h}(\a)$ satisfies Formula~\ref{FormulaAHMinimumNumberOfSteps}.

   It remains to show that \sh is 1-minimal constructible. Since any element of the partition can not be constructed in one step from the remaining others, we have that if $A_i\in\h$ then  \,$\h\not\sub C_1\left(\h-\{A_i\}\right).$\,  Hence, by  Lemma~\ref{lmmremove1element},  $\h$ is 1-minimal constructible.
\end{proof}

\begin{theorem}\label{thmonnmc}
   Let \sa be a finite algebra of subsets of $X$. Then there exists a family $\h$ of subsets of $X$ such that
   \begin{itemize}
      \item $\h \sub \a$
      \item For all $n\ge 1, \ \,\h$ is $n$-minimal constructible
      \item $\h_{\infty} = \a$
   \end{itemize}
\end{theorem}

\begin{proof}
   As in the previous proof, consider  the partition \,$\p=\{A_1,A_2,...,A_k\}$\, induced by the equivalence relationship given in Lemma~\ref{LemmaEquivalenceRelationship}. If $k=1$, clearly \,$\h=\p$\, satisfies the conclusion of the theorem just as in the previous proof as any proper subfamily of \sh is empty. Similarly, if $k=2,$ consider \,$\h=\p-\{A_2\}.$ Clearly \sh satisfies \,$\h \sub \a, \, \hi = \p_\infty=\a$ and \sh is also $n$-minimal constructible as any proper subfamily is empty. Suppose $k>2$\, and consider \,$\h=\p-\{A_k\}.$ Clearly \sh satisfies \,$\h \sub \a$\, and  \,$\hi = \p_\infty.$\,

   It only remains to check that \sh is $n$-minimal constructible.  Consider $A_i\in\h,\ x\in A_i$ and $y\in A_k.$ Suppose that $\h\sub C_n\left(\h-\{A_i\}\right).$ \, It is not hard to see that if $B_1,...,B_{k-2}$ are the sets in $\p-\{A_i,A_k\}$ and $B_{k-1}=A_i\cup A_k$ then $\{B_1,B_2,...,B_{k-1}\}$ is a partition of $X$ and
   \begin{align}\label{EquationCinftyPartition}
      C_\infty\left(\h-\{A_i\}\right)=C_\infty\left(\p-\{A_i,A_k\}\right)=C_\infty\left(\{B_1,B_2,...,B_{k-1}\}\right).
   \end{align}
   Since $x,y$ are unseparable in $\{B_1,B_2,...,B_{k-1}\}$,  Lemma~\ref{LemmaCharactizeElementsAlgebraPartition} implies that $x,y$ are unseparable in $C_\infty\left(\{B_1,B_2,...,B_{n-1}\right\}$.\, Equation~\ref{EquationCinftyPartition} implies that $x,y$ are unseparable in $C_\infty\left(\h-\{A_i\}\right)$ and hence they are unseparable in a smaller family such as $C_n\left(\h-\{A_i\}\right)$ or even \shh; however, clearly $x,y$ are separable in $\h$ as $x\in A_i, y\notin A_i$ and $A_i\in\h.$ Therefore the relationship $\h\sub C_n\left(\h-\{A_i\}\right)$ is impossible. Lemma~\ref{lmmremove1element} shows now that \sh is $n$-minimal constructible.
\end{proof}

%%%%%%%%%%%%%%%%%%%%%%%%%%%%%%%%%%%%%%%%%%%%%%%%%%%%%%%%%%%%%%%%%%%%%%%%%%%%%%%%%%%%%%%%%%%%%%%%%%%%%%%%%%%%%%%%%%%%%%%%%%%%%%%%%%%%%%%%%%%%%%%%%%%%%%%%%%%%%%%%%%%%%%%%%%%%%%%%%%%%%%%%%%%%%%%%%%%%%%%%%%%%%%%%%%%%%%%%%%%%%%%%%%%%%%%%%%%%%%%%%%%%%%%%%%%%%%%%%%%%%%%%%%%%%%%%%%%%%%%%%%%%%%%%%%%%%%%%%%%%%%%%%%%%%%%%%%%%%%%%%%%%%%

\section{Future Work and Open Questions} %%%%%%%%%%%  666666666666666666666666666
\label{SectionFutureWork}

   \indent There are a number of applications to research on constructible sets. In particular, \cite{Aprena2011} has examined some uses of the properties of set algebras in the context of economics; in this work, information can be modeled in terms of families of sets, while the exchange of information takes place through the mechanisms of union, intersection and complement. It is also possible that there exist applications in the areas of data storage and computer science; much of our work has been on generating large, complicated set algebras from small, simple families. This idea of finding simple generating families may have some use in data compression; a family representing some information could be replaced by a family which is 1- or $n$-minimal constructible.\\
   \indent There remain a number of open questions and possible areas of future work raised by our research. In particular, we have not arrived at a general formula for counting the number of $n$-minimal constructible families for a given finite universe.  Connections between this counting problem and the conjecture of Frankl (see \cite{Bosnjak2008}) ought to be examined as well.\\
   \indent In addition to continuing to examine the finite case, more work in infinite families is necessary. Separability is a property that needs to be studied in the infinite case. In the finite case separability implies \,$\ui=2^X$ but in the infinite case, this is not the case because even though $\R$ is separable in \,$\u=$ usual topology, we do not have \,$\ui=2^\R$ since $\Q\notin\ui.$ In the upcoming paper~\cite{GaBaMoBoPoCe} we   count the number of $1$-minimal constructible families and we  compute the number of $n$-minimal constructible families under some restrictions; however, a general formula needs to be found. Counting the families \sh of smallest size that are $n$-minimal constructible and generate an algebra \sa is also an interesting problem to be approached.


\begin{thebibliography}{99}
    \bibitem{Allouche1996} \textsc{Allouche, J-P.} (1996). \textit{Note on the Constructible Sets of a
        Topological Space}, Annals of the New York Academy of Sciences, v 806, Issue 1, p1-10. 
    \bibitem{Billingsley1995} \textsc{Billingsley, P.} (1995). \textit{Probability and Measure}, 3rd ed. Wiley, New York.
    \bibitem{Grinblat2010} \textsc{Grinblat, L.} \v{S}. \textit{Algebras of Sets and Combinatorics.} Providence, RI: American Mathematical Society, 2002. Print.
    \bibitem{RoydenFitzpatrick2010} \textsc{Royden, H. L. \& Patrick Fitzpatrick.} \textit{Real Analysis.} Boston: Prentice Hall, 2010. Print.
    \bibitem{Folland1999} \textsc{Folland,~G.}  \textit{Real Analysis:  Modern Techniques and Their Applications}(2nd ed.).  Wiley-Interscience, 1999.
    \bibitem{GaBaMoBoPoCe} \textsc{Garcia J.,Ballesteros P., Polin M., Morales W., Bongers T., Cervantes J.}, \textit{On Decomposition of Constructible Sets}, In Progress.
    \bibitem{Bosnjak2008} \textsc{Ivica, Bo\v{s}njak, \& Markovic Petar.} \textit{The 11-element Case of Frankl's Conjecture.} The Electronic Journal of Combinatorics 15 (2008). Web. 4 Aug. 2011. $<$http://www.combinatorics.org$>$.
    \bibitem{Marker2002} \textsc{Marker D.} \textit{Model Theory: An Introduction}. Springer. 2002.
    \bibitem{Aprena2011} \textsc{Apreda, Rodolfo.} \textit{Differential Rates of Return and Residual Information Sets} UCEMA $|$ Universidad Del CEMA. Web. 04 Aug. 2011. $<$http://www.ucema.edu.ar$>$.
    \bibitem{Comtet1994} \textsc{Comtet, L.}, \textit{Advanced Combinatorics}, D. Reidel, Boston, 1974.
 \end{thebibliography}
\end{document}